\documentclass[11pt,reqno]{amsart}
\usepackage{amsthm}
\usepackage{a4,amsmath,amssymb,graphicx,amscd,xy}
\usepackage[usenames,dvipsnames]{color}
\usepackage{graphicx}
\usepackage{longtable}
\usepackage{mathtools}
\usepackage{color}
\usepackage{verbatim}
\usepackage{pict2e}
\usepackage[utf8,utf8x]{inputenc}
\usepackage{enumitem}
\usepackage{longtable}
\usepackage{hyperref}
\allowdisplaybreaks

\parskip=\smallskipamount

\begin{document}

\author{Lars Simon}
\address{Lars Simon, Department of Mathematical Sciences, Norwegian University of Science and Technology, Trondheim, Norway}
\email{lars.simon@ntnu.no}

\author{Berit Stens\o nes}
\address{Berit Stens\o nes, Department of Mathematical Sciences, Norwegian University of Science and Technology, Trondheim, Norway}
\email{berit.stensones@ntnu.no}

\thanks{The second author is supported by the Research Council of Norway, Grant number 240569/F20.}
\thanks{This work was done during the international research program "Several Complex Variables and Complex Dynamics" at the Centre for Advanced Study at the Academy of Science and Letters in Oslo during the academic year 2016/2017.}
\title{On Newton Diagrams of Plurisubharmonic Polynomials}

%
%

\subjclass[2010]{Primary 32T25.  Secondary 32C25.}
\keywords{Bumping, plurisubharmonic polynomial, Newton diagram, finite-type domain, extreme edge.}

\begin{abstract}
Each extreme edge of the Newton diagram of a plurisubharmonic polynomial on $\mathbb{C}^2$ gives rise to a plurisubharmonic polynomial. It is tempting to believe that the union of the extreme edges or the convex hull of said union will do the same. We construct a plurisubharmonic polynomial $P$ on $\mathbb{C}^2$ with precisely two extreme edges $E_1$ and $E_2$, such that neither $E_1\cup{E_2}$ nor $\text{Conv}({E_1\cup{}E_2})$ yields a plurisubharmonic polynomial.
\end{abstract}

\maketitle

\section{Introduction}\label{introduction}
It is a well-known fact that it is possible to solve the $\overline{\partial}$-equation with supnorm estimates for sufficiently regular $\overline{\partial}$-closed $(0,1)$-forms on bounded strictly pseudoconvex domains in $\mathbb{C}^n$ with boundary of class $\mathcal{C}^2$. This was shown by H.\ Grauert and I.\ Lieb \cite{GrauertLieb} and G.M.\ Henkin \cite{Henkin} in the case of higher boundary regularity and by N.\ {\O}vrelid \cite{Ovrelid} for boundaries of class $\mathcal{C}^2$.

If, however, $\Omega\subseteq\mathbb{C}^n$ is a bounded weakly pseudoconvex domain with boundary of class $\mathcal{C}^{\infty}$, it is not necessarily possible to solve the $\overline{\partial}$-equation with supnorm estimates. In fact, N.\ Sibony \cite{Sibony} has constructed a bounded weakly pseudoconvex domain $D\subseteq\mathbb{C}^3$ with $\mathcal{C}^{\infty}$-boundary which admits a $\overline{\partial}$-closed $(0,1)$-form $\Phi\in\mathcal{C}_{0,1}^{\infty}(D)\cap\mathcal{C}_{0,1}^{0}(\overline{D})$, such that the equation $\overline{\partial}\Psi=\Phi$ has no bounded solution on $D$.

It hence becomes and interesting question which additional assumptions on a bounded weakly pseudoconvex domain $\Omega\subseteq\mathbb{C}^n$ with smooth boundary guarantee the existence of supnorm estimates for solutions of $\overline{\partial}u=f$, where $f$ is a sufficiently regular $\overline{\partial}$-closed $(0,1)$-form on $\Omega$.\\
R.M.\ Range \cite{Range} has shown that supnorm (and even H{\"o}lder) estimates {\emph{do}} exist for bounded smoothly bounded pseudoconvex domains of finite type in $\mathbb{C}^2$. Later K.\ Diederich, B.\ Fischer and J.E.\ Forn{\ae}ss \cite{DiederichFischerFornaess} obtained estimates for bounded smoothly bounded convex domains of finite type in $\mathbb{C}^n$.

One of the crucial ingredients in Range's argument is the {\emph{local bumping}} of the domain at a boundary point. Following \cite{BharaliStensones}, one defines a local bumping of a smoothly bounded pseudoconvex domain $\Omega\subseteq\mathbb{C}^n$, $n\geq{2}$, at a boundary point $\zeta\in\partial\Omega$ to be a triple $(\partial\Omega{},U_{\zeta},\rho_{\zeta})$, such that:
\begin{itemize}
\item{$U_{\zeta}\subseteq\mathbb{C}^n$ is an open neighborhood or $\zeta$,}
\item{$\rho_{\zeta}\colon{}U_{\zeta}\to\mathbb{R}$ is smooth and plurisubharmonic,}
\item{$\rho_{\zeta}^{-1}(\{0\})$ is a smooth hypersurface in $U_{\zeta}$ that is pseudoconvex from the side $U_{\zeta}^{-}:=\{z\colon{}\rho_{\zeta}(z)<0\}$,}
\item{$\rho_{\zeta}(\zeta)=0$, but $\rho_{\zeta}<0$ on $U_{\zeta}\cap\left(\overline{\Omega}\setminus{\{\zeta\}}\right)$.}
\end{itemize}
Given a bounded smoothly bounded pseudoconvex domain $D$ of finite type in $\mathbb{C}^2$, Range proceeds by producing a bumping $D_p$ of $D$ at a boundary point $p\in\text{b}D$, fitting large polydiscs centered in $D$ into $D_p$ and thus obtaining good pointwise estimates for holomorphic functions using the Cauchy estimates. This in turn he uses to construct integral kernels for the $\overline{\partial}$-equation satisfying the necessary estimates. The finite type condition is necessary to ensure that the above-mentioned polydiscs are large enough.

When the dimension is increased, however, it becomes much harder to construct local bumpings of the domain. For the remainder of this section let $\Omega\subseteq\mathbb{C}^n$, $n\geq{2}$, be a bounded pseudoconvex domain with real-analytic boundary. In this situation, K.\ Diederich and J.E.\ Forn{\ae}ss have shown in \cite{DiederichFornaess2} that local bumpings always exist at each boundary point. This, however, is a priori not enough to construct good integral kernels and hence obtain supnorm or H{\"o}lder estimates for $\overline{\partial}$, since the order of contact between $\partial\Omega$ and the boundary of the bumped out domain at a boundary point $p\in\partial\Omega$ can be a lot higher than the type of the domain $\Omega$ when $n\geq{3}$.\\
The goal hence becomes to construct a local bumping of $\Omega$ at a boundary point $p\in\partial\Omega$, such that the order of contact between $\partial\Omega$ and the boundary of the bumped out domain at $p$ does not exceed the type of the domain in any direction. It should be noted that $\Omega$ is of finite type, as was shown by K.\ Diederich and J.E.\ Forn{\ae}ss \cite{DiederichFornaess}.

So let $p$ be a boundary point of $\Omega$. After a holomorphic change of coordinates one can assume that $p=0$ and that the domain is given as follows:
\begin{align*}
\phantom{=} & \Omega\cap{V}\\
= & \{(\zeta,z)\in{}(\mathbb{C}\times\mathbb{C}^{n-1})\cap{V}\colon{}\operatorname{Re}(\zeta)+r(z)+\mathcal{O}(|\operatorname{Im}(\zeta)|^2,|z|\cdot{}|\operatorname{Im}(\zeta)|)<0\}\text{,}
\end{align*}
where $V$ is a small open neighborhood of $p=0$ and $r$ is a real-valued real-analytic function defined on an open neighborhood of $0\in\mathbb{C}^{n-1}$. Furthermore $r$ can be chosen to be of the form
\begin{align*}
r(z)=\sum_{j=2k}^{\infty}{P_j(z)}\text{,}
\end{align*}
where $P_j$ is a homogeneous polynomial in $z$ and $\overline{z}$ of degree $j$ and $P_{2k}\not\equiv{0}$ (i.e.\ the lowest-degree term of $r$ has degree $2k$, which is less or equal to the type of $\Omega$ at $p=0$) and $P_{2k}$ is plurisubharmonic but not pluriharmonic. In the special case $\Omega\subseteq\mathbb{C}^2$ one can show that it is possible to find such a local description, such that $2k$ is actually equal to the type of the domain at $p=0$. By absorbing all pluriharmonic terms of $P_{2k}$ into the real part of $\zeta$, one can assume that $P_{2k}$ has no pluriharmonic terms.\\
When $\Omega\subseteq\mathbb{C}^2$, J.E.\ Forn{\ae}ss and N.\ Sibony \cite{FornaessSibony} have shown that the domain can be bumped to order $2k$, the type of the domain. Further A.\ Noell \cite{MR1207878} showed that if $P_{2k}$ is additionally assumed to not be harmonic along any complex line through $0\in\mathbb{C}^{n-1}$ this is still the case. But if $P_{2k}$ is allowed to be harmonic along complex lines through $0$, things become much more complicated.

Noell proceeded by showing that there exist an $\mathbb{R}$-homogeneous function $\widetilde{P}_{2k}\colon\mathbb{C}^{n-1}\to\mathbb{R}$ of degree $2k$ and a constant $\epsilon{}>0$, such that
\begin{align*}
P_{2k}(z)-\widetilde{P}_{2k}(z)\geq{}\epsilon{}|z|^{2k}\text{ for all }z\in\mathbb{C}^{n-1}\text{,}
\end{align*}
and such that $\widetilde{P}_{2k}$ is smooth and strictly plurisubharmonic on $\mathbb{C}^{n-1}\setminus{\{0\}}$.

The next step is to look for similar results without assuming $P_{2k}$ to not be harmonic along any complex line through $0$. In this case, however, one can not expect to obtain an inequality as strong as the one in Noell's result, since that would lead to a violation of the strong maximum principle for subharmonic functions along a complex line through $0$ along which $P_{2k}$ is harmonic (i.e.\ vanishes, since $P_{2k}$ does not have any pluriharmonic terms). A similar argument also shows that one can not expect to get something {\em{strictly}} plurisubharmonic on $\mathbb{C}^{n-1}\setminus{\{0\}}$.

Assume $n=3$ for the remainder of this section. In this situation G.\ Bharali and B.\ Stens{\o}nes \cite{BharaliStensones} have obtained bumping results for the polynomial $P_{2k}\colon\mathbb{C}^2\to\mathbb{R}$ in two different cases. They prove that that $P_{2k}$ is harmonic along at most finitely many complex lines through $0$, which, in one of the two cases, allows them to combine local bumpings in conical neighborhoods of said lines using a gluing argument.\\
Since $P_{2k}$ can be harmonic along complex lines through $0$, however, this does not necessarily lead to a bumping of the domain $\Omega$. This paper deals with the problem of finding a bumping for the domain $\Omega$ in the case $n=3$ and provides a counterexample to a proposed strategy.

\section{Motivating Examples}\label{motivatingexamples}
Let $\Omega$ be a bounded pseudoconvex domain with real-analytic boundary in $\mathbb{C}^3$ and $p\in\partial{\Omega}$. As in the introduction, after a holomorphic change of coordinates, one can assume that $p=0$ and that
\begin{align*}
\phantom{=} & \Omega\cap{V}\\
= & \{(\zeta,z,w)\in{}\mathbb{C}^{3}\cap{V}\colon{}\operatorname{Re}(\zeta)+r(z,w)+\mathcal{O}(|\operatorname{Im}(\zeta)|^2,|(z,w)|\cdot{}|\operatorname{Im}(\zeta)|)<0\}\text{,}
\end{align*}
where $V$ is a small open neighborhood of $p=0$ and $r$ is a real-valued real-analytic function defined on an open neighborhood of $0\in\mathbb{C}^{2}$. Since this paper is on a counterexample, we limit ourselves to the case where $r$ is a plurisubharmonic polynomial. By absorbing all pluriharmonic terms into the real part of $\zeta$, one can assume that $r$ has no pluriharmonic terms. Write
\begin{align*}
r(z,w)=\sum_{j=2k}^{M}{P_j(z,w)}\text{,}
\end{align*}
where $P_j$ is a homogeneous polynomial in $z,\overline{z},w,\overline{w}$ of degree $j$ and $P_{2k}\not\equiv{0}$ is plurisubharmonic.

If the remainder
\begin{align*}
R(z,w):=r(z,w)-P_{2k}(z,w)=\sum_{j=2k+1}^{M}{P_j(z,w)}
\end{align*}
is plurisubharmonic then a bumping with the desired properties exists in many cases. The situation is not usually that simple however, so a different strategy is needed when the remainder $R$ is not assumed to be plurisubharmonic.

\allowdisplaybreaks[0]

\theoremstyle{definition}
\newtheorem{erstesbeispiel}[propo]{Example}
\begin{erstesbeispiel}
\label{erstesbeispiel}
Assume $\Omega$ is given as follows locally around $0$:
\begin{align*}
\Omega\cap{V}=\{(\zeta,z,w)\in{}\mathbb{C}^{3}\cap{V}\colon{}\operatorname{Re}(\zeta)+P(z,w)<0\}\text{,}
\end{align*}
where
\begin{align*}
P(z,w)= & |z|^6|w|^{8}-2\operatorname{Re}(z^3w^4\overline{z^5w^3})+|z|^4|w|^{12}+|z|^{10}|w|^{6}-2\operatorname{Re}(zw^{10}\overline{z^2w^6})\\
& +|z|^{18}|w|^{4}+|z|^{2}|w|^{20}-2\operatorname{Re}(z^9w^{2}\overline{z^{17}w})+|z|^{34}|w|^{2}+\left\lVert{(z,w)}\right\rVert^{1000}\text{.}
\end{align*}
Define singular holomorphic coordinate changes $\Phi_1,\Phi_2,\Phi_3\colon\mathbb{C}^2\to\mathbb{C}^2$ by
\begin{align*}
\Phi_1(z,w) & =\left(z^4,w\right)\text{,}\\
\Phi_2(z,w) & =\left(z,w^2\right)\text{,}\\
\Phi_3(z,w) & =\left(z,w^8\right)\text{.}
\end{align*}
We compute:
\begin{align*}
\left({}P\circ\Phi_1\right)(z,w) & =|z|^8|w|^{20}-2\operatorname{Re}(z^4w^{10}\overline{z^8w^6})+|z|^{16}|w|^{12}\\
& \phantom{=}+(higher\text{-}order\text{ }terms)\\
& =\left|z^{4}w^{10}-z^{8}w^{6}\right|^2+(higher\text{-}order\text{ }terms)\text{,}\\
\left({}P\circ\Phi_2\right)(z,w) & =|z|^6|w|^{16}-2\operatorname{Re}(z^3w^{8}\overline{z^5w^6})+|z|^{10}|w|^{12}\\
& \phantom{=}+(higher\text{-}order\text{ }terms)\\
& =\left|z^{3}w^{8}-z^{5}w^{6}\right|^2+(higher\text{-}order\text{ }terms)\text{,}\\
\left({}P\circ\Phi_3\right)(z,w) & =|z|^{18}|w|^{32}-2\operatorname{Re}(z^9w^{16}\overline{z^{17}w^8})+|z|^{34}|w|^{16}\\
& \phantom{=}+(higher\text{-}order\text{ }terms)\\
& =\left|z^{9}w^{16}-z^{17}w^{8}\right|^2+(higher\text{-}order\text{ }terms)\text{.}
\end{align*}
For $j\in\{1,2,3\}$, the lowest-order homogeneous summand of $P\circ\Phi_j$ corresponds to the summand $P^{(j)}$ in the Taylor expansion of $P$ around $0$, where
\begin{align*}
P^{(1)}(z,w) & =\left|zw^{10}-z^{2}w^{6}\right|^2\text{,}\\
P^{(2)}(z,w) & =\left|z^{3}w^{4}-z^{5}w^{3}\right|^2\text{,}\\
P^{(3)}(z,w) & =\left|z^{9}w^{2}-z^{17}w\right|^2\text{.}
\end{align*}
$P^{(1)}$, $P^{(2)}$ and $P^{(3)}$ are plurisubharmonic. This is not a coincidence: $P$ is plurisubharmonic and $\Phi_j$, $j\in\{1,2,3\}$, is holomorphic, so the lowest order homogeneous summand of $P\circ\Phi_j$ is plurisubharmonic as well, which (despite $\Phi_j$ being a {\emph{singular}} holomorphic coordinate change) leads to $P^{(j)}$ being plurisubharmonic. $P^{(1)}$, $P^{(2)}$ and $P^{(3)}$ have pairwise no monomial in common, so:
\begin{align*}
P=P^{(1)}+P^{(2)}+P^{(3)}+(remaining\text{ }terms)\text{,}
\end{align*}
where the $(remaining\text{ }terms)$ consists of a finite (possibly empty) sum of monomials, each appearing with the same coefficient as the corresponding monomial in the Taylor expansion of $P$ around $0$. By direct computation one easily verifies that
\begin{align*}
P(z,w)=P^{(1)}(z,w)+P^{(2)}(z,w)+P^{(3)}(z,w)+\left\lVert{(z,w)}\right\rVert^{1000}\text{.}
\end{align*}
So we have written $P$ as a sum of four plurisubharmonic weighted-homogeneous polynomials. It is obvious how to bump $P$. In a more general setting one could attempt to use the bumping results for weighted-homogeneous plurisubharmonic polynomials in \cite{BharaliStensones} to bump each summand separately.
\end{erstesbeispiel}

\theoremstyle{definition}
\newtheorem{zweitesbeispiel}[propo]{Example}
\begin{zweitesbeispiel}
\label{zweitesbeispiel}
Assume $\Omega$ is given as follows locally around $0$:
\begin{align*}
\Omega\cap{V}=\{(\zeta,z,w)\in{}\mathbb{C}^{3}\cap{V}\colon{}\operatorname{Re}(\zeta)+P(z,w)<0\}\text{,}
\end{align*}
where
\begin{align*}
P(z,w) & =|z|^6-2\operatorname{Re}(z^3\overline{z^2w^2})+2|z|^4|w|^{4}\\
& \phantom{{}=}-2\operatorname{Re}(z^2w^2\overline{w^{10}})+|w|^{20}+\left\lVert{(z,w)}\right\rVert^{1000}\text{.}
\end{align*}
Analogously to Example \ref{erstesbeispiel}, one defines singular holomorphic coordinate changes $\Phi_1,\Phi_2\colon\mathbb{C}^2\to\mathbb{C}^2$ by
\begin{align*}
\Phi_1(z,w) & =\left(z^2,w\right)\text{,}\\
\Phi_2(z,w) & =\left(z^4,w\right)\text{,}
\end{align*}
and computes:
\begin{align*}
\left({}P\circ\Phi_1\right)(z,w) & =|z|^{12}-2\operatorname{Re}(z^6\overline{z^4w^2})+2|z|^{8}|w|^{4}+(higher\text{-}order\text{ }terms)\text{,}\\
\left({}P\circ\Phi_2\right)(z,w) & =2|z|^{16}|w|^{4}-2\operatorname{Re}(z^8w^{2}\overline{w^{10}})+|w|^{20}+(higher\text{-}order\text{ }terms)\text{.}
\end{align*}
For $j\in\{1,2\}$, the lowest-order homogeneous summand of $P\circ\Phi_j$ corresponds to the summand $P^{(j)}$ in the Taylor expansion of $P$ around $0$, where
\begin{align*}
P^{(1)}(z,w) & =|z|^6-2\operatorname{Re}(z^3\overline{z^2w^2})+2|z|^4|w|^{4}\text{,}\\
P^{(2)}(z,w) & =2|z|^4|w|^{4}-2\operatorname{Re}(z^2w^2\overline{w^{10}})+|w|^{20}\text{.}
\end{align*}
Analogously to the previous example, one argues that $P^{(1)}$ and $P^{(2)}$ are plurisubharmonic. But now the polynomials $P^{(1)}$ and $P^{(2)}$ share the summand $2|z|^4|w|^{4}$, so one can {\emph{not}} proceed analogously to Example \ref{erstesbeispiel}.\\
Splitting up the shared summand, however, one can write:
\begin{align*}
P(z,w)=\widetilde{P}^{(1)}(z,w)+\widetilde{P}^{(2)}(z,w)+\left\lVert{(z,w)}\right\rVert^{1000}\text{,}
\end{align*}
where
\begin{align*}
\widetilde{P}^{(1)}(z,w) & =|z|^6-2\operatorname{Re}(z^3\overline{z^2w^2})+|z|^4|w|^{4}\\
& =\left|z^{3}-z^{2}w^2\right|^2\text{,}\\
\widetilde{P}^{(2)}(z,w) & =|z|^4|w|^{4}-2\operatorname{Re}(z^2w^2\overline{w^{10}})+|w|^{20}\\
& =\left|z^{2}w^{2}-w^{10}\right|^2\text{.}
\end{align*}
$\widetilde{P}^{(1)}$ and $\widetilde{P}^{(2)}$ are obviously plurisubharmonic and hence we have once again written $P$ as a sum of plurisubharmonic weighted-homogeneous polynomials, each of which we can attempt to bump individually.
\end{zweitesbeispiel}

\allowdisplaybreaks

So, in both Example \ref{erstesbeispiel} and Example \ref{zweitesbeispiel}, we used certain singular holomorphic coordinate changes to express $P$ as a sum of weighted-homogeneous plurisubharmonic polynomials. While the algorithmic procedure we applied will not always yield such a decomposition, the existence of said coordinate changes is not a coincidence: in both examples, each coordinate change corresponds to an {\emph{extreme edge}} (see Def.\ \ref{extremeedge} below) of the real-valued plurisubharmonic polynomial $P$.

\section{The Problem}\label{theproblem}
Most of the definitions and lemmas in this section are taken from \cite{FornStens2010}. From now on, all occurring polynomials are assumed to be polynomials with complex coefficients in two complex variables $(z,w)$ and their conjugates $(\overline{z},\overline{w})$.

Let $P$ be a real-valued polynomial. We write
\begin{align*}
P=\sum_{(A,B)\in\mathbb{Z}_{\geq{0}}\times\mathbb{Z}_{\geq{0}}}^{}{P_{A,B}}\text{,}
\end{align*}
where $P_{A,B}$ is homogeneous of degree $A$ in $z, \overline{z}$ and homogeneous of degree $B$ in $w, \overline{w}$. Note that this decomposition is unique and that each $P_{A,B}$ is real-valued.

\theoremstyle{definition}
\newtheorem{extremeedge}[propo]{Definition}
\begin{extremeedge}
\label{extremeedge}
Let $P$ be a real-valued polynomial. We define the {\emph{Newton diagram}} $N(P)$ of $P$ to be the following subset of $\mathbb{R}^2$:
\begin{align*}
N(P)=\{(A,B)\in\mathbb{Z}_{\geq{0}}\times\mathbb{Z}_{\geq{0}}\colon{}P_{A,B}\not\equiv{0}\}\text{.}
\end{align*}
We make the following definitions:
\begin{itemize}
\item{A non-empty subset $X\subseteq{}N(P)$ is called an {\emph{extreme set}} if there exist $a,b\in\mathbb{R}$ with $a<0$, such that
\begin{align*}
& B=aA+b\text{ for all }(A,B)\in{}X\\
& B>aA+b\text{ for all }(A,B)\in{}N(P)\setminus{}X\text{.}
\end{align*}}
\item{A point $(A_0,B_0)\in{N(P)}$ is called an {\emph{extreme point}} if $\{(A_0,B_0)\}$ is an extreme set.}
\item{A subset $E\subseteq{}N(P)$ is called an {\emph{extreme edge}} if $E$ is an extreme set of cardinality at least $2$.}
\end{itemize}
\end{extremeedge}

\theoremstyle{definition}
\newtheorem{polysetindex}[propo]{Notation}
\begin{polysetindex}
\label{polysetindex}
Let $P$ be a real-valued polynomial and let $S\subseteq\mathbb{R}^2$. We define the real-valued polynomial $P_S$ as follows:
\begin{align*}
P_S:=\sum_{(A,B)\in{}N(P)\cap{}S}^{}{P_{A,B}}\text{.}
\end{align*}
Note that $P_S\equiv{0}$ if and only if $N(P)\cap{}S=\emptyset$.
\end{polysetindex}

\theoremstyle{definition}
\newtheorem{complehessiannotation}[propo]{Notation}
\begin{complehessiannotation}
\label{complehessiannotation}
Let $P$ be a real-valued polynomial. We denote the Complex Hessian Matrix or the Levi Matrix of $P$ as $H_P$,
\begin{align*}
\arraycolsep=0.2pt\def\arraystretch{1.4}
H_P=\left( \begin{array}{ccc}
\frac{\partial^2 P}{\partial{z}\partial{\overline{z}}} & \frac{\partial^2 P}{\partial{w}\partial{\overline{z}}} \\
\frac{\partial^2 P}{\partial{z}\partial{\overline{w}}} & \frac{\partial^2 P}{\partial{w}\partial{\overline{w}}} \end{array} \right)\text{.}
\end{align*}
\end{complehessiannotation}

The following two lemmas demonstrate that the concepts introduced in this section are significant when considering plurisubharmonic polynomials:

\theoremstyle{plain}
\newtheorem{finitelymanyextremeedges}[propo]{Lemma}
\begin{finitelymanyextremeedges}
\label{finitelymanyextremeedges}
Let $P$ be a real-valued polynomial. Then the Newton diagram $N(P)$ has finitely many extreme sets.
\end{finitelymanyextremeedges}

\theoremstyle{plain}
\newtheorem{extremeedgegivespsh}[propo]{Lemma}
\begin{extremeedgegivespsh}
\label{extremeedgegivespsh}
Let $P$ be a real-valued polynomial and furthermore assume that $P$ is plurisubharmonic. Then, for any extreme set $X$ of $N(P)$, the function $P_X$ is a plurisubharmonic weighted-homogeneous polynomial and there exists a natural singular holomorphic change of coordinates $\Phi$ of the form $(z,w)\mapsto (z^k,w^l)$ with $k,l\in\mathbb{Z}_{\geq 1}$, $\gcd(k,l)=1$, such that $P_X\circ\Phi$ constitutes the lowest-order homogeneous terms of $P\circ\Phi$.
\end{extremeedgegivespsh}

In the setting of Example \ref{erstesbeispiel}, the maps $\Phi_1$, $\Phi_2$ and $\Phi_3$ correspond to extreme edges, say $E_1$, $E_2$ and $E_3$, of $N(P)$ in the sense of Lemma \ref{extremeedgegivespsh} (it should be noted, however, that $N(P)$ has other extreme edges as well). Since $E_1$, $E_2$ and $E_3$ are pairwise disjoint, the polynomials $P_{E_1}$, $P_{E_2}$ and $P_{E_3}$ have pairwise no terms in common, so that
\begin{align*}
P_{E_1\cup E_2\cup E_3}=P_{E_1}+P_{E_2}+P_{E_3}
\end{align*}
is plurisubharmonic and
\begin{align*}
P(z,w)=P_{E_1}(z,w)+P_{E_2}(z,w)+P_{E_3}(z,w)+\left\lVert{(z,w)}\right\rVert^{1000}\text{.}
\end{align*}

In the setting of Example \ref{zweitesbeispiel}, the maps $\Phi_1$ and $\Phi_2$ correspond to the precisely two extreme edges, say $E_1$ and $E_2$, of $N(P)$ in the sense of Lemma \ref{extremeedgegivespsh}. Here, however, $E_1$ and $E_2$ are neighboring extreme edges, so that $P_{E_1}$ and $P_{E_2}$ have terms in common, namely $P_{E_1\cap E_2}$. But $P_{E_1\cup E_2}$ is plurisubharmonic and we found a splitting
\begin{align*}
P_{E_1\cup E_2}=\widetilde{P_{E_1}}+\widetilde{P_{E_2}}\text{,}
\end{align*}
where $\widetilde{P_{E_j}}$ is a plurisubharmonic polynomial with $N\left(\widetilde{P_{E_j}}\right)\subseteq N(P_{E_j})$, for $j\in\{1,2\}$.

In attempting to generalize the bumping strategies outlined in Examples \ref{erstesbeispiel} and \ref{zweitesbeispiel}, it becomes desirable to identify subsets of the Newton diagram of a plurisubharmonic polynomial that will yield a plurisubharmonic function in the sense of Notation \ref{polysetindex}. It is the content of Lemma \ref{extremeedgegivespsh} that extreme sets, i.e.\ extreme points and extreme edges, are examples of such subsets.\\
Specifically, it has been suspected that two ``neighboring'' extreme edges would yield a plurisubharmonic function by taking their union or by taking the convex hull of that union. A precise statement of those questions goes as follows:

\theoremstyle{definition}
\newtheorem{problemofthispaper}[propo]{Question}
\begin{problemofthispaper}
\label{problemofthispaper}
Let $P$ be a real-valued polynomial and furthermore assume that $P$ is plurisubharmonic. Let $\mathcal{E}$ denote the (possibly empty) set of extreme edges of $N(P)$.
\begin{itemize}
\item{Given extreme edges $E_1$ and $E_2$ of $N(P)$ with $E_1\neq{}E_2$ but $E_1\cap{}E_2\neq\emptyset$, is $P_{E_1\cup{}E_2}$ necessarily plurisubharmonic in some neighborhood of the origin?}
\item{Given extreme edges $E_1$ and $E_2$ of $N(P)$ with $E_1\neq{}E_2$ but $E_1\cap{}E_2\neq\emptyset$, is $P_{\text{Conv}({E_1\cup{}E_2})}$ necessarily plurisubharmonic in some neighborhood of the origin?}
\item{Is $P_{{\bigcup_{E\in\mathcal{E}}{E}}}$ necessarily plurisubharmonic in some neighborhood of the origin?}
\item{Is $P_{\text{Conv}({\bigcup_{E\in\mathcal{E}}{E}})}$ necessarily plurisubharmonic in some neighborhood of the origin?}
\end{itemize}
Here, $\text{Conv}(S)$ denotes the convex hull of a subset $S$ of $\mathbb{R}^2$.
\end{problemofthispaper}

In the following section we will construct a plurisubharmonic polynomial with precisely $2$ extreme edges, for which the answer to all of these questions is ``no''.

\section{The Counterexample}\label{thesolution}
In order to simplify the computations in the construction announced in the previous section, we state and prove the following lemma:

\theoremstyle{plain}
\newtheorem{detofcomplhess}[propo]{Lemma}
\begin{detofcomplhess}
\label{detofcomplhess}
Let $P=\sum_{\alpha\in\mathcal{A}}^{}{c_{\alpha}\cdot{}\left\vert{f_{\alpha}}\right\vert}^2$, where
\begin{itemize}
\item{$\mathcal{A}$ is a finite set,}
\item{$c_{\alpha}\in\{-1,1\}$ for all $\alpha\in\mathcal{A}$,}
\item{$f_{\alpha}\colon{}\mathbb{C}^2\to\mathbb{C}$ is a holomorphic polynomial for all $\alpha\in\mathcal{A}$.}
\end{itemize}
Then in $\mathbb{C}^2$ we have:
\begin{align*}
\det{H_P}=\frac{1}{2}\cdot\sum_{(\alpha{,}\beta)\in{}\mathcal{A}\times\mathcal{A}}^{}{c_{\alpha}{}c_{\beta}\left\vert{\frac{\partial{f_{\alpha}}}{\partial{z}}\cdot{}\frac{\partial{f_{\beta}}}{\partial{w}}-\frac{\partial{f_{\beta}}}{\partial{z}}\cdot{}\frac{\partial{f_{\alpha}}}{\partial{w}}}\right\vert^2}\text{.}
\end{align*}
\end{detofcomplhess}

\begin{proof}
We calculate:
\begin{align*}
\det{H_P} & =\left(\sum_{\alpha\in\mathcal{A}}^{}{c_{\alpha}\frac{\partial{f_{\alpha}}}{\partial{z}}}{\frac{\partial{\overline{f_{\alpha}}}}{\partial{\overline{z}}}}\right)\cdot\left(\sum_{\beta\in\mathcal{A}}^{}{c_{\beta}\frac{\partial{f_{\beta}}}{\partial{w}}}{\frac{\partial{\overline{f_{\beta}}}}{\partial{\overline{w}}}}\right)\\
& \phantom{=}-\left(\sum_{\alpha\in\mathcal{A}}^{}{c_{\alpha}\frac{\partial{f_{\alpha}}}{\partial{z}}}{\frac{\partial{\overline{f_{\alpha}}}}{\partial{\overline{w}}}}\right)\cdot\left(\sum_{\beta\in\mathcal{A}}^{}{c_{\beta}\frac{\partial{f_{\beta}}}{\partial{w}}}{\frac{\partial{\overline{f_{\beta}}}}{\partial{\overline{z}}}}\right)\\
& =\left(\sum_{\alpha\in\mathcal{A}}^{}{c_{\alpha}\frac{\partial{f_{\alpha}}}{\partial{z}}}\overline{\left(\frac{\partial{f_{\alpha}}}{\partial{z}}\right)}\right)\cdot\left(\sum_{\beta\in\mathcal{A}}^{}{c_{\beta}\frac{\partial{f_{\beta}}}{\partial{w}}}\overline{\left(\frac{\partial{f_{\beta}}}{\partial{w}}\right)}\right)\\
& \phantom{=}-\left(\sum_{\alpha\in\mathcal{A}}^{}{c_{\alpha}\frac{\partial{f_{\alpha}}}{\partial{z}}}\overline{\left(\frac{\partial{f_{\alpha}}}{\partial{w}}\right)}\right)\cdot\left(\sum_{\beta\in\mathcal{A}}^{}{c_{\beta}\frac{\partial{f_{\beta}}}{\partial{w}}}\overline{\left(\frac{\partial{f_{\beta}}}{\partial{z}}\right)}\right)\\
& =\sum_{(\alpha{},\beta)\in\mathcal{A}\times\mathcal{A}}^{}{c_{\alpha}c_{\beta}\cdot\frac{\partial{f_{\alpha}}}{\partial{z}}\cdot{}\frac{\partial{f_{\beta}}}{\partial{w}}\cdot\overline{\left({\frac{\partial{f_{\alpha}}}{\partial{z}}\cdot{}\frac{\partial{f_{\beta}}}{\partial{w}}-\frac{\partial{f_{\beta}}}{\partial{z}}\cdot{}\frac{\partial{f_{\alpha}}}{\partial{w}}}\right)}}\\
& =\frac{1}{2}\cdot{}\sum_{(\alpha{},\beta)\in\mathcal{A}\times\mathcal{A}}^{}{c_{\alpha}c_{\beta}\cdot\frac{\partial{f_{\alpha}}}{\partial{z}}\cdot{}\frac{\partial{f_{\beta}}}{\partial{w}}\cdot\overline{\left({\frac{\partial{f_{\alpha}}}{\partial{z}}\cdot{}\frac{\partial{f_{\beta}}}{\partial{w}}-\frac{\partial{f_{\beta}}}{\partial{z}}\cdot{}\frac{\partial{f_{\alpha}}}{\partial{w}}}\right)}}\\
& \phantom{=}+\frac{1}{2}\cdot{}\sum_{(\beta{},\alpha)\in\mathcal{A}\times\mathcal{A}}^{}{c_{\beta}c_{\alpha}\cdot\frac{\partial{f_{\beta}}}{\partial{z}}\cdot{}\frac{\partial{f_{\alpha}}}{\partial{w}}\cdot\overline{\left({\frac{\partial{f_{\beta}}}{\partial{z}}\cdot{}\frac{\partial{f_{\alpha}}}{\partial{w}}-\frac{\partial{f_{\alpha}}}{\partial{z}}\cdot{}\frac{\partial{f_{\beta}}}{\partial{w}}}\right)}}\\
& =\frac{1}{2}\cdot\sum_{(\alpha{},\beta)\in\mathcal{A}\times\mathcal{A}}^{}{c_{\alpha}c_{\beta}}
\!\begin{aligned}[t] & {\cdot\left(\frac{\partial{f_{\alpha}}}{\partial{z}}\cdot{}\frac{\partial{f_{\beta}}}{\partial{w}}-\frac{\partial{f_{\beta}}}{\partial{z}}\cdot{}\frac{\partial{f_{\alpha}}}{\partial{w}}\right)}\\
& {\cdot\overline{\left({\frac{\partial{f_{\alpha}}}{\partial{z}}\cdot{}\frac{\partial{f_{\beta}}}{\partial{w}}-\frac{\partial{f_{\beta}}}{\partial{z}}\cdot{}\frac{\partial{f_{\alpha}}}{\partial{w}}}\right)}}
\end{aligned}
\nonumber\\
& =\frac{1}{2}\cdot\sum_{(\alpha{,}\beta)\in{}\mathcal{A}\times\mathcal{A}}^{}{c_{\alpha}{}c_{\beta}\left\vert{\frac{\partial{f_{\alpha}}}{\partial{z}}\cdot{}\frac{\partial{f_{\beta}}}{\partial{w}}-\frac{\partial{f_{\beta}}}{\partial{z}}\cdot{}\frac{\partial{f_{\alpha}}}{\partial{w}}}\right\vert^2}\text{.}
\end{align*}
\end{proof}

Let $f_1,f_2,f_3,g,h\colon\mathbb{C}^2\to\mathbb{C}$ be the holomorphic monomials given as follows:
\begin{align*}
f_1(z,w) & =z^2w^2  & f_2(z,w) & =z^{10}w & f_3(z,w)=zw^{10}\\
g(z,w) & =z^4w^2 & h(z,w) & =z^4w^8
\end{align*}
We now define a real-valued polynomial $P$:
\begin{align*}
P:=\left\vert{f_1+f_2+f_3}\right\vert^2{+}\left\vert{g+h}\right\vert^2\text{.}
\end{align*}
It is obvious that $P$ is plurisubharmonic. Intuitively speaking, the Newton diagram $N(P)$ has precisely two extreme edges and lies entirely in the triangle spanned by $N(\vert{}f_1\vert^2)$, $N(\vert{}f_2\vert^2)$ and $N(\vert{}f_3\vert^2)$, with the exception of $N(\vert{}h\vert^2)$, which is ``peaking out'' of the triangle without creating an extreme edge. Both extreme edges correspond to sides of said triangle. The monomials were specifically chosen to have these properties (among others). We will treat this formally:

\theoremstyle{plain}
\newtheorem{newtonofP}[propo]{Lemma}
\begin{newtonofP}
\label{newtonofP}
The Newton diagram of $P$ is the following set:
\begin{align*}
N(P)=\{(4,4),(12,3),(3,12),(20,2),(11,11),(2,20),(8,4),(8,10),(8,16)\}\text{.}
\end{align*}
Furthermore, $N(P)$ has precisely two extreme edges, namely
\begin{align*}
E_1=\{(4,4),(3,12),(2,20)\}\text{ and }E_2=\{(4,4),(12,3),(20,2)\}\text{,}
\end{align*}
and the following holds on $\mathbb{C}^2$:
\begin{align*}
P_{{E_1\cup{}E_2}} & = \vert{f_1+f_3}\vert^2+\vert{f_1+f_2}\vert^2-\vert{f_1}\vert^2\text{,}\\
P_{\text{{\emph{Conv}}}({E_1\cup{}E_2})} & = P-\vert{h}\vert^2\\
& = \left\vert{f_1+f_2+f_3}\right\vert^2{+}\left\vert{g+h}\right\vert^2-\vert{h}\vert^2\text{.}
\end{align*}
\end{newtonofP}

The proof of Lemma \ref{newtonofP} is a straightforward calculation and will be omitted. It should, however, be remarked that, in light of Lemma \ref{detofcomplhess}, the monomials occurring in the definition of $P$ were chosen so that $P_{{E_1\cup{}E_2}}$ and $P_{\text{Conv}({E_1\cup{}E_2})}$ take this particular form.\ \\

In order to show that (for $P$) the answer to all the questions in Question \ref{problemofthispaper} is ``no'', it suffices to show that both $P_{{E_1\cup{}E_2}}$ and $P_{\text{Conv}({E_1\cup{}E_2})}$ are {\emph{not}} plurisubharmonic in any neighborhood of the origin.\\
By Lemma \ref{detofcomplhess} and Lemma \ref{newtonofP} we have the following on $\mathbb{C}^2$:
\begin{align*}
\det{H_{P_{{E_1\cup{}E_2}}}} & =\phantom{-}\left\vert{\frac{\partial{(f_1+f_3)}}{\partial{z}}\cdot{}\frac{\partial{(f_1+f_2)}}{\partial{w}}-\frac{\partial{(f_1+f_2)}}{\partial{z}}\cdot{}\frac{\partial{(f_1+f_3)}}{\partial{w}}}\right\vert^2{}\\
& \phantom{={}}-\left\vert{\frac{\partial{(f_1+f_3)}}{\partial{z}}\cdot{}\frac{\partial{f_1}}{\partial{w}}-\frac{\partial{f_1}}{\partial{z}}\cdot{}\frac{\partial{(f_1+f_3)}}{\partial{w}}}\right\vert^2{}\\
& \phantom{={}}-\left\vert{\frac{\partial{(f_1+f_2)}}{\partial{z}}\cdot{}\frac{\partial{f_1}}{\partial{w}}-\frac{\partial{f_1}}{\partial{z}}\cdot{}\frac{\partial{(f_1+f_2)}}{\partial{w}}}\right\vert^2\\
& \leq\phantom{-}\left\vert{\frac{\partial{(f_1+f_3)}}{\partial{z}}\cdot{}\frac{\partial{(f_1+f_2)}}{\partial{w}}-\frac{\partial{(f_1+f_2)}}{\partial{z}}\cdot{}\frac{\partial{(f_1+f_3)}}{\partial{w}}}\right\vert^2{}\\
& \phantom{={}}-\left\vert{\frac{\partial{(f_1+f_3)}}{\partial{z}}\cdot{}\frac{\partial{f_1}}{\partial{w}}-\frac{\partial{f_1}}{\partial{z}}\cdot{}\frac{\partial{(f_1+f_3)}}{\partial{w}}}\right\vert^2{}\text{,}\\
\det{H_{P_{\text{Conv}({E_1\cup{}E_2})}}} & =\phantom{-}\left\vert{\frac{\partial{(f_1+f_2+f_3)}}{\partial{z}}\cdot{}\frac{\partial{(g+h)}}{\partial{w}}-\frac{\partial{(g+h)}}{\partial{z}}\cdot{}\frac{\partial{(f_1+f_2+f_3)}}{\partial{w}}}\right\vert^2{}\\
& \phantom{={}}-\left\vert{\frac{\partial{(f_1+f_2+f_3)}}{\partial{z}}\cdot{}\frac{\partial{h}}{\partial{w}}-\frac{\partial{h}}{\partial{z}}\cdot{}\frac{\partial{(f_1+f_2+f_3)}}{\partial{w}}}\right\vert^2{}\\
& \phantom{={}}-\left\vert{\frac{\partial{(g+h)}}{\partial{z}}\cdot{}\frac{\partial{h}}{\partial{w}}-\frac{\partial{h}}{\partial{z}}\cdot{}\frac{\partial{(g+h)}}{\partial{w}}}\right\vert^2\\
& \leq\phantom{-}\left\vert{\frac{\partial{(f_1+f_2+f_3)}}{\partial{z}}\cdot{}\frac{\partial{(g+h)}}{\partial{w}}-\frac{\partial{(g+h)}}{\partial{z}}\cdot{}\frac{\partial{(f_1+f_2+f_3)}}{\partial{w}}}\right\vert^2{}\\
& \phantom{={}}-\left\vert{\frac{\partial{(g+h)}}{\partial{z}}\cdot{}\frac{\partial{h}}{\partial{w}}-\frac{\partial{h}}{\partial{z}}\cdot{}\frac{\partial{(g+h)}}{\partial{w}}}\right\vert^2{}\text{.}
\end{align*}

So, by plugging in and calculating, we get the following inequalities on $\mathbb{C}^2$:
\begin{align*}
& \phantom{\leq}\det{H_{P_{{E_1\cup{}E_2}}}(z,w)}\\
& \leq\phantom{-}\left\vert{(2zw^2+w^{10})\cdot{}(2z^2w+z^{10})-(2zw^2+10z^9w)\cdot{}(2z^2w+10zw^9)}\right\vert^2\\
& \phantom{\leq{}}-\left\vert{(2zw^2+w^{10})\cdot{}2z^2w-2zw^2\cdot{}(2z^2w+10zw^9)}\right\vert^2\\
& =\phantom{-}\left\vert{z^2w^2(99z^8w^8+18z^9+18w^9)}\right\vert^2\\
& \phantom{\leq{}}-\left\vert{18z^2w^{11}}\right\vert^2\text{,}
\end{align*}
and
\begin{align*}
& \phantom{\leq}\det{H_{P_{\text{Conv}({E_1\cup{}E_2})}}(z,w)}\\
& \leq\phantom{-}\vert{(2zw^2+10z^9w+w^{10})\cdot{}(2z^4w+8z^4w^7)}\\
& \phantom{\leq{}-\vert}{-(4z^3w^2+4z^3w^8)\cdot{}(2z^2w+z^{10}+10zw^9)}\vert^2\\
& \phantom{\leq{}}-\left\vert{(4z^3w^2+4z^3w^8)\cdot{8z^4w^7}-4z^3w^8\cdot{}(2z^4w+8z^4w^7)}\right\vert^2\\
& =\phantom{-}\left\vert{-2z^4w^2(16w^{15}-38w^6z^9+19w^9-8z^9-4zw^7+2zw)}\right\vert^2\\
& \phantom{\leq{}}-\left\vert{24z^7w^9}\right\vert^2\text{.}
\end{align*}

We define two holomorphic polynomials $Q_1,Q_2\colon\mathbb{C}^2\to\mathbb{C}$ as follows:
\begin{align*}
Q_1(z,w) & =99z^8w^8+18z^9+18w^9\text{,}\\
Q_2(z,w) & =16w^{15}-38w^6z^9+19w^9-8z^9-4zw^7+2zw\text{,}
\end{align*}
i.e.\ we have on $\mathbb{C}^2$:
\begin{align*}
\det{H_{P_{{E_1\cup{}E_2}}}(z,w)} & \leq\phantom{-}\left\vert{z^2w^2Q_1(z,w)}\right\vert^2\\
& \phantom{\leq{}}-\left\vert{18z^2w^{11}}\right\vert^2\text{,}\\
\det{H_{P_{\text{Conv}({E_1\cup{}E_2})}}(z,w)} & \leq\phantom{-}\left\vert{-2z^4w^2Q_2(z,w)}\right\vert^2\\
& \phantom{\leq{}}-\left\vert{24z^7w^9}\right\vert^2\text{.}
\end{align*}

Since $Q_1$ is a non-constant holomorphic polynomial on $\mathbb{C}^2$, its vanishing set $V(Q_1)$ is an equidimensional affine algebraic variety of dimension $1$ containing $(0,0)$. For $(z,w)\in{}V(Q_1)$ we have
\begin{align*}
\det{H_{P_{{E_1\cup{}E_2}}}(z,w)}\leq{}-\left\vert{18z^2w^{11}}\right\vert^2\text{,}
\end{align*}
so that it suffices to show that $V(Q_1)$ contains points $(z,w)$ with $z\neq{}0,w\neq{}0$ arbitrarily close to $(0,0)$. But that is clear, since both $Q_1(\cdot,0)$ and $Q_1(0,\cdot)$ are non-constant holomorphic polynomials on $\mathbb{C}$ and as such have finitely many zeroes.

Hence $P_{{E_1\cup{}E_2}}$ is not plurisubharmonic in any neighborhood or the origin. By considering $Q_2$ instead of $Q_1$, we analogously get that $P_{\text{Conv}({E_1\cup{}E_2})}$ is not plurisubharmonic in any neighborhood of the origin.

\bibliographystyle{amsplain}
\bibliography{refs}

\end{document}